\documentclass[number,sort&compress]{elsarticle}

\usepackage{amssymb,amsmath,amsthm,url,placeins}
\usepackage[hidelinks]{hyperref}
\usepackage{cleveref}

\crefname{equation}{}{}
\Crefname{equation}{Equation}{Equations}
\crefname{theorem}{Theorem}{Theorems}
\Crefname{theorem}{Theorem}{Theorems}
\crefname{lemma}{Lemma}{Lemmas}
\Crefname{lemma}{Lemma}{Lemmas}
\crefname{proposition}{Proposition}{Propositions}
\Crefname{proposition}{Proposition}{Propositions}
\crefname{corollary}{Corollary}{Corollaries}
\Crefname{corollary}{Corollary}{Corollaries}
\crefname{conjecture}{Conjecture}{Conjectures}
\Crefname{conjecture}{Conjecture}{Conjectures}
\crefname{section}{Section}{Sections}
\Crefname{section}{Section}{Sections}
\crefname{example}{Example}{Examples}
\Crefname{example}{Example}{Examples}
\crefname{problem}{Problem}{Problems}
\Crefname{problem}{Problem}{Problems}
\crefname{table}{Table}{Tables}
\Crefname{table}{Table}{Tables}
\crefname{remark}{Remark}{Remarks}
\Crefname{remark}{Remark}{Remarks}
\crefname{definition}{Definition}{Definitions}
\Crefname{definition}{Definition}{Definitions}

\newcommand{\ZZ}{\mathbb{Z}}

\newcommand{\s}{\mathsf{s}}

\newtheorem{theorem}{Theorem}

\theoremstyle{definition}

\newcommand{\doi}[1]{\href{http://dx.doi.org/#1}{\texttt{doi:#1}}}

\title{Modified {E}rd\H{o}s--{G}inzburg--{Z}iv constants for $\ZZ_2^{\lowercase{d}}$
}

\author{Alexander Sidorenko}
\ead{sidorenko.ny@gmail.com}
\address{R\'{e}nyi Institute, Budapest, Hungary}

\date{\today}

\begin{document}

\begin{abstract}
Let $G$ be a finite abelian group written additively, 
and let $r$ be a multiple of its exponent. 
The modified Erd\H{o}s--Ginzburg--Ziv constant $\s_r'(G)$
is the smallest integer $s$
such that every zero-sum sequence of length $s$ over $G$ 
has a zero-sum subsequence of length $r$. 
We find exact values of $\s_{2k}'(\mathbb{Z}_2^d)$ for $d \leq 2k+1$. 
\end{abstract}


\maketitle


Let $G$ be a finite abelian group written additively. 
We denote by $\exp(G)$ the {\it exponent} of $G$
that is the least common multiple of the orders of its elements. 
Let $r$ be a multiple of $\exp(G)$. 
The \emph{generalized Erd\H{o}s--Ginzburg--Ziv constant} $\s_r(G)$ 
is the smallest integer $s$
such that every sequence of length $s$ over $G$ 
has a zero-sum subsequence of length $r$. 
If $r = \exp(G)$, then $\s(G)=\s_{\exp(G)}(G)$ 
is the classical Erd\H{o}s--Ginzburg--Ziv constant. 
The constants $\s_r(G)$ have been studied extensively, 
see for example 
\cite{Bitz:2020,Gao:2003,Gao:2014,Gao:2006,Han:2018,Han:2019,He:2016,Sidorenko:2020}. 
The following variation of these constants was introduced in \cite{Augspurger:2017} 
and further studied in \cite{Berger:2019,Berger:2019b,Hu:2023}. 
The \emph{modified Erd\H{o}s--Ginzburg--Ziv constant} $\s_r'(G)$ 
is the smallest integer $s$
such that every \emph{zero-sum} sequence of length $s$ over $G$ 
has a zero-sum subsequence of length $r$. 

By the definition, $\s_r'(G) \leq \s_r(G)$. 
On the other hand, if $g_1,g_2,\ldots,g_s$ is a sequence over $G$
that does not contain a zero-sum subsequence of size $r$, 
and $s$ is mutually prime with $\exp(G)$, 
then there exists $x \in G$ such that 
$g_1+x,g_2+x,\ldots,g_s+x$ is a zero-sum subsequence 
(see \cite{Augspurger:2017,Hu:2023}).  
Thus, $\s_r'(G) \geq \s_r(G) - (\exp(G) - 1)$,
and if $\s_r(G)-1$ is mutually prime with $\exp(G)$, 
then $\s_r'(G) = \s_r(G)$. 

In this note, we consider the case $\exp(G)=2$, so $G\cong\ZZ_2^d$. 
By the abovementioned argument,
\begin{equation}\label{eq:equal}
  \s_r'(\ZZ_2^d) = \s_r(\ZZ_2^d) \;\;\;{\rm if}\; 
    \s_r(\ZZ_2^d) \;{\rm is\; even},
\end{equation}
and 
\begin{equation}\label{eq:ineq}
  \s_r(\ZZ_2^d)-1 \:\leq\: \s_r'(\ZZ_2^d) \:\leq\: \s_r(\ZZ_2^d).
\end{equation}
The exact values of generalized Erd\H{o}s--Ginzburg--Ziv constants 
$\s_{2k}(\ZZ_2^d)$ have been found for $d \leq 2k+1$: 

\begin{theorem}[\cite{Sidorenko:2020}]\label{th:T1}
\[
  \s_{2k}(\ZZ_2^d) = \begin{cases}
  2k+d \;\;\;{\rm for}\;\; d < 2k; \\
  4k+1 \;\;\;{\rm for}\;\; d = 2k; \\
  4k+2 \;\;\;{\rm for}\;\; d = 2k+1,\;\; k \;{\rm is\; even}; \\
  4k+5 \;\;\;{\rm for}\;\; d = 2k+1,\;\; k \;{\rm is\; odd}.
  \end{cases}
\]
\end{theorem}

In the present note, we extend this result to the \emph{modified} 
Erd\H{o}s--Ginzburg--Ziv constants. 

\begin{theorem}\label{th:T2}
Let $d \leq 2k+1$. 
Then $\s_{2k}'(\ZZ_2^d) = \s_{2k}(\ZZ_2^d)-1$ in the following cases: 
\begin{itemize}
\item 
$d=2k-1$;
\item 
$d=2k-3$, $k$ is even;
\item 
$d \leq 2k-5$, $d$ is odd.
\end{itemize}
In all other cases, $\s_{2k}'(\ZZ_2^d) = \s_{2k}(\ZZ_2^d)$.
\end{theorem}

\begin{proof}[\bf{Proof}]
We start with the cases where we claim 
$\s_{2k}'(\ZZ_2^d) = \s_{2k}(\ZZ_2^d)$. 
Among them, cases $d < 2k$ with even $d$, and $d=2k+1$ with even $k$ 
follow from \cref{th:T1} and \cref{eq:equal}. 
The other three cases are 
$d=2k$, $d=2k+1$ with odd $k$, and $d=2k-3$ with odd $k$. 
Since $\s_{2k}'(\ZZ_2^d) \leq \s_{2k}(\ZZ_2^d)$, 
it is sufficient to construct 
a zero-sum sequence of length $\s_{2k}(\ZZ_2^d)-1$ 
that does not contain a zero-sum subsequence of length $2k$. 
For $d=2k$, we select a sequence of length $4k$ which consists of 
$2k-1$ copies of the zero vector, 
the $2k$ basis vectors $e_1,e_2,\ldots,e_{2k}$, 
and the vector $e_1+e_2+\ldots+e_{2k}$. 
For odd $k$ and $d=2k+1,2k-3$, 
we select a sequence of length $2d+2$ which consists of 
$0,\,e_1,\,e_2,\,\ldots,e_{d-1},\,e_1+e_2+\ldots+e_{d-1},\,
e_d,\,e_d+e_1,\,e_d+e_2,\,\ldots,e_d+e_{d-1},\,e_d+e_1+e_2+\ldots+e_{d-1}$. 

To solve the three cases where we claim 
$\s_{2k}'(\ZZ_2^d) = \s_{2k}(\ZZ_2^d)-1$, 
in the light of \cref{eq:ineq}, it is sufficient to prove that 
any zero-sum sequence of length $\s_{2k}(\ZZ_2^d)-1$ over $\ZZ_2^d$ 
contains a zero-sum subsequence of length $2k$. 
First consider the case $d=2k-1$. 
Let $x_2,x_3,\ldots,x_{4k-1} \in \ZZ_2^{2k-1}$ 
where $x_2+x_3+\ldots+x_{4k-1} = 0$. 
Set $x_1=x_2$. 
As $\s_{2k}(\ZZ_2^{2k-1}) = 4k-1$, 
there is $A\subset\{1,2,\ldots,4k-1\}$ such that $|A|=2k$ 
and $\sum_{i \in A} x_i = 0$. 
If $1 \notin A$, then we have found a zero-sum subsequence 
of length $2k$ among $x_2,x_3,\ldots,x_{4k-1}$. 
Suppose, $1 \in A$. 
If $2 \notin A$, then $(A \backslash \{1\}) \cup \{2\}$ 
points to a zero-sum subsequence of length $2k$. 
Suppose, $1,2 \in A$. 
Set $B := (\{1,2,\ldots,4k-1\} \backslash A) \cup \{2\}$. 
Then $|B|=2k$ and 
\begin{align*}
  \sum_{i \in B} x_i & = x_1+x_2+\ldots+x_{4k-1} - \sum_{i \in A} x_i + x_2
  \\ & = x_1+x_2 + (x_2+\ldots+x_{4k-1}) - \sum_{i \in A} x_i = x_1+x_2 = 0.     
\end{align*}
Finally, let $d$ be odd, 
and $d \leq 2k-3$ if $k$ is even, 
or $d \leq 2k-5$ if $k$ is odd. 
We are going to show that 
every zero-sum sequence of length $2k+d-1$ over $\ZZ_2^d$ 
contains a zero-sum subsequence of length $2k$. 
Let $x_1,x_2,\ldots,x_{2k+d-1}\in\ZZ_2^d$
where $x_1+x_2+\ldots+x_{2k+d-1}=0$. 
By \cref{th:T1}, 
$s_{d-1}(\ZZ_2^d)=2d$ if $d \equiv 1 \;{\rm mod}\; 4$, and
$s_{d-1}(\ZZ_2^d)=2d+3$ if $d \equiv 3 \;{\rm mod}\; 4$. 
In both cases, 
$s_{d-1}(\ZZ_2^d) \leq 2k+d-1$. 
Thus, there is $A \subset \{1,2,\ldots,2k+d-1\}$ such that 
$|A|=d-1$ and $\sum_{i \in A} x_i = 0$. 
Set $B := \{1,2,\ldots,2k+d-1\} \backslash A$. 
Then $|B|=2k$ and 
$\sum_{i \in B} x_i = \sum_{i=1}^{2k+d-1} x_i - \sum_{i \in A} x_i = 0 - 0 = 0$,
so $B$ points to a zero-sum subsequence of length $2k$ within 
$x_1,\ldots,x_{2k+d-1}$. 
\end{proof}

\end{document}